\documentclass[12pt, reqno]{amsart}
\usepackage{hyperref}
\usepackage[utf8]{inputenc}
\usepackage{graphicx}
\usepackage{amsmath}
\usepackage{amsthm}
\usepackage{amssymb,bbm}%
\usepackage[numbers, square]{natbib}
\usepackage{color}

\theoremstyle{plain}
\newtheorem{theorem}{Theorem}[section]
\newtheorem{lemma}[theorem]{Lemma}
\newtheorem{corollary}[theorem]{Corollary}
\newtheorem{proposition}[theorem]{Proposition}

\theoremstyle{definition}

\theoremstyle{remark}
\newtheorem{remark}[theorem]{Remark}

\makeindex
\makeglossary

\newcommand{\cF}{\mathcal{F}}

\newcommand{\cQ}{\mathcal{Q}}
\newcommand{\cP}{\mathcal{P}}

\newcommand{\E}{\mathbb E\,}
\newcommand{\R}{\mathbb{R}}
\newcommand{\N}{\mathbb{N}}

\newcommand{\s}{\mathbb{S}}
\newcommand{\B}{\mathbb{B}}
\renewcommand{\P}{\mathbb{P}}

\newcommand{\Int}{\mathop{\mathrm{int}}\nolimits}
\newcommand{\cl}{\mathop{\mathrm{cl}}\nolimits}

\newcommand{\relint}{\mathop{\mathrm{relint}}\nolimits}

\newcommand{\image}{\mathop{\mathrm{Im}}\nolimits}

\newcommand{\Vol}{\mathop{\mathrm{Vol}}\nolimits}

\newcommand{\conv}{\mathop{\mathrm{conv}}\nolimits}
\newcommand{\pos}{\mathop{\mathrm{pos}}\nolimits}

\newcommand{\lin}{\mathop{\mathrm{lin}}\nolimits}
\newcommand{\aff}{\mathop{\mathrm{aff}}\nolimits}
\newcommand{\dist}{\mathop{\mathrm{dist}}\nolimits}

\newcommand{\argmin}{\mathop{\mathrm{argmin}}\nolimits}
\newcommand{\argmax}{\mathop{\mathrm{argmax}}\nolimits}

\newcommand{\eqdistr}{\stackrel{d}{=}}

\newcommand{\ind}{\mathbbm{1}}

\begin{document}

\author[F.~G\"otze]{Friedrich G\"otze}
\address{Friedrich G\"otze, Faculty of Mathematics, Bielefeld University, P. O. Box 10 01 31, 33501 Bielefeld, Germany}
\email{goetze@math.uni-bielefeld.de}

\author[Z.~Kabluchko]{Zakhar Kabluchko}
\address{Zakhar Kabluchko, Institute of Mathematical Stochastics, Orl\'eans-Ring 10,
48149 M\"unster, Germany}
\email{zakhar.kabluchko@uni-muenster.de}

\author[D.~Zaporozhets]{Dmitry Zaporozhets}
\address{Dmitry Zaporozhets, St.~Petersburg Department of Steklov Institute of Mathematics, Fontanka~27, 191011 St.~Petersburg, Russia
Russia}
\email{zap1979@gmail.com}

\title[Grassmann angles of Gaussian convex hulls] {Grassmann angles  and absorption probabilities of Gaussian convex hulls}
\date{\today}
\keywords{Conic intrinsic volumes, persistence probability, conic Crofton formula, conic Steiner formula, Sudakov's formula, Tsirelson's formula, Grassmann angle, Gaussian image, absorption probability, Gaussian simplex}
\subjclass[2010]{Primary: 52A22, 60D05; Secondary: 52A55, 52B11.}
\thanks{The work of the first and third authors was done with the financial support of the Bielefeld University (Germany) in terms of project SFB 701. The work of the third author is supported by the grant RFBR
16-01-00367 and by the Program of Fundamental
Researches of Russian Academy of Sciences ``Modern Problems of
Fundamental Mathematics''. The second author was supported by the German Research Foundation under Germany's Excellence Strategy  EXC 2044 -- 390685587, Mathematics M\"unster: Dynamics - Geometry - Structure.}

\begin{abstract}
Let $M$ be an arbitrary subset in $\R^n$ with a conic (or positive)
hull $C$. Consider its Gaussian image $AM$, where  $A$ is a $k\times
n$-matrix whose entries are independent standard Gaussian random
variables. We show that the probability that the convex hull of $AM$
contains the origin in its interior coincides with the $k$-th
Grassmann angle of $C$. Also, we prove that the expected Grassmann
angles of $AC$ coincide with the corresponding Grassmann angles of
$C$.  Using the latter result, we show that the expected sum of $j$-th
Grassmann angles at $\ell$-dimensional faces of a Gaussian simplex
equals the analogous angle-sum for the regular simplex of the same
dimension.
\end{abstract}

\maketitle

\section{Introduction}
\subsection{Introduction}\label{subsec:intro}
Consider a random linear operator $A:\R^n\to\R^k$ whose matrix, also denoted by $A$, is given by
\begin{equation}
A:=\left(
\begin{array}{ccc}
N_{11}&\dots &N_{1n}\\
\vdots &\cdots&\vdots\\
N_{k1}&\dots &N_{kn}\\
\end{array}
\right)\in\R^{k\times n},
\end{equation}
where $N_{11},\dots, N_{kn}$ are independent standard Gaussian random variables.  If $M$ is any subset of $\R^n$, then the set
\[
AM:=\{Ax:x\in M\}\subset \R^k
\]
is called the \textit{Gaussian image} (or spectrum) of $M$. Sometimes it is also called the Gaussian projection of $M$ even though $A^2\neq A$ a.s.  On the other hand, let  $W_k$  be a  random $k$-dimensional linear subspace of $\R^n$ distributed uniformly on the Grassmannian of all such subspaces, and let $M|W_k$ denote the orthogonal projection of $M$ on $W_k$, where $k\in \{1,\ldots,n\}$.
In stochastic geometry, there are several results that relate geometric characteristics of $AM$  to the corresponding characteristics of $M|W_k$.
Probably the most elegant example is a theorem of Baryshnikov and Vitale~\cite{baryshnikov_vitale} which states\footnote{In~\cite{baryshnikov_vitale}, the result is stated for $P$ being a so-called $SR$-simplex  or a cube, but the same proof applies to every convex polytope.} that for a  convex polytope $P\subset \R^n$ and for any affine-invariant functional $\varphi$ defined on the set of all polytopes, $\varphi(AP)$ has the same distribution as $\varphi(P|W_k)$.
For example, the joint distributions of the $f$-vectors of $AP$ and $P|W_k$ coincide, namely
\begin{equation}
(f_0 (A P),\ldots, f_k(A P) ) \eqdistr (f_0 (P|W_k),\ldots, f_k (P|W_k)),
\end{equation}
where $f_\ell(P)$ denotes the number of $\ell$-dimensional faces of $P$.
In particular, we have the equality of expectations:
\begin{equation}\label{eq:bar_vit_2}
\E f_\ell (A P) =  \E f_\ell (P|W_k),
\end{equation}
for all $\ell \in \{0,\ldots,k\}$.
Another result of this type, relating the expected volume of $AM$ to the expected volume of $M|W_k$ for convex $M$, was obtained by Sudakov~\cite{vS79} and Tsirelson~\cite{bT85}; see Section~\ref{1239}.  In the present paper,  we shall study Gaussian images of convex cones and, in particular,  compute their expected Grassmann angles. Even though Grassmann angles are not affine-invariant, it will turn out that an analogue of Baryshnikov--Vitale theorem holds: For every convex cone $C\subset \R^n$, the expected Grassmann angles of $AC$ coincide with the expected Grassmann angles of the random orthogonal projection $C|W_k$. But first we shall recall the formulae of Sudakov and Tsirelson in Section~\ref{1239} after collecting the necessary definitions in Section~\ref{1608}. For an extensive account on stochastic, convex and integral geometry we refer to the books~\cite{SW08,SchneiderBook,MR1608265}.

\subsection{Intrinsic volumes}\label{1608}
Let $K$ be a  compact convex subset of $\R^n$. The basic geometric characteristics of $K$ are the \emph{intrinsic volumes} $V_0(K), \ldots, V_n(K)$ which are defined as the coefficients in the Steiner formula
\begin{equation}\label{eq:steiner}
\Vol_n(K+r\B^n)=\sum_{k=0}^n \kappa_{n-k} V_k(K) r^{n-k}, \quad r\geq 0.
\end{equation}
Here, we write $\Vol_k(\cdot)$ for the $k$-dimensional volume (Lebesgue measure), $\B^k$ denotes the $k$-dimensional unit ball, and $\kappa_k:=\Vol_k(\B^k)=\pi^{k/2}/\Gamma(\frac k 2 +1)$ its volume. By definition, 
$\kappa_0=1$.

Recall that $W_k$  is  a  random linear $k$-plane  in $\R^n$ distributed according to the Haar measure. An equivalent way to define the intrinsic volumes is the Kubota formula:
\begin{equation}\label{2110}
V_k(K)= \binom nk\frac{\kappa_n}{\kappa_k \kappa_{n-k}}\E\Vol_k(K|W_k),
\end{equation}
where we recall that $K|W_k$ is the orthogonal projection of $K$ on $W_k$.
It is known that  $V_n(\cdot)$ coincides with the $n$-dimensional volume,  $V_{n-1}(\cdot)$ is half the  surface area, and $V_1(\cdot)$  coincides with the mean width, up to a constant factor.

\begin{remark}\label{645}
The normalization  constants in~\eqref{eq:steiner} and~\eqref{2110} are chosen in such a way that the intrinsic volumes of a set do not depend on the dimension of the ambient space: if we embed $K$ into $\R^N$ with $N\geq n$, the intrinsic volumes remain the same.
\end{remark}


\subsection{Gaussian representation of intrinsic volumes}\label{1239}
Using the rotation invariance property of the Gaussian distribution, Sudakov (for $k=1$) and Tsirelson (for all $k$) found the  Gaussian analogue of the Kubota formula. They also generalized it to the infinite-dimensional Hilbert space. In the present paper, we do not need this level of generalization. So, for simplicity, we state their results in the finite-dimensional case only.

Let ${N_1},\dots,{N_n}\in\R^1$ be independent standard Gaussian variables. In~\cite{vS79}, Sudakov showed that for any  compact set $M\subset\R^n$,
\begin{equation}\label{2041}
V_1(\conv M)=\sqrt{2\pi}\,\E\sup_{(t_1,\dots,t_n)\in M}\sum_{i=1}^{n}t_i{N_i},
\end{equation}
where $\conv M$ denotes the convex hull of $M$.
Tsirelson~\cite{bT85} generalized this result to all intrinsic volumes as follows:
\begin{equation}\label{2042}
V_k( \conv M )= \frac{(2\pi)^{k/2}}{k!\kappa_k} \E\,\Vol_k(\conv AM),
\end{equation}
for all $k=0,1,\dots,n$.
In view of~\eqref{2110}, this formula establishes a relation between the expected volume of the uniform projection and the expected volume of the Gaussian projection of $\conv M$. A simple and elegant proof of~\eqref{2042} can be found in~\cite{vitale08}.
To see that~\eqref{2042} generalizes~\eqref{2041}, note that in the case when $k=1$ we have
\begin{align*}
\conv A M
&=
\conv\Big\{\sum_{i=1}^{n}t_i{N_{i}}:(t_1,\dots,t_n)\in M\Big\}\\
&=
\Big[\min_{(t_1,\dots,t_n)\in M}\sum_{i=1}^{n}t_i{N_{i}},\max_{(t_1,\dots,t_n)\in M}\sum_{i=1}^{n}t_i{N_{i}}\Big],
\end{align*}
where we write $N_i:=N_{1i}$.
Hence, for $k=1$, \eqref{2042} reduces to
$$
V_1(\conv M )=\sqrt{\frac\pi 2}\, \E \Big[\max_{(t_1,\dots,t_n)\in M}\sum_{i=1}^{n}t_i{N_i}-\min_{(t_1,\dots,t_n)\in M}\sum_{i=1}^{n}t_i{N_i}\Big].
$$
However, since  $(\sum_{i=1}^{n}t_i{N_i})$ and $(-\sum_{i=1}^{n}t_i{N_i})$, considered as Gaussian processes indexed by $M$,  have the same distribution, it holds that
\[
\E\min_{(t_1,\dots,t_n)\in M}\sum_{i=1}^{n}t_i{N_i} = -\E\max_{(t_1,\dots,t_n)\in M}\sum_{i=1}^{n}t_i{N_i}
\]
and we recover~\eqref{2041}.

The basic example of~\eqref{2042} is when $M$ is the set of the standard orthonormal  basis vectors in $\R^n$. Then $\conv M$ is the regular $n-1$-dimensional simplex and $AM$ is a set of $n$ independent standard Gaussian vectors in $\R^k$. Thus, \eqref{2042} relates the expected volume of the Gaussian polytope (defined as the convex hull of $AM$) with the intrinsic volumes of the regular simplex, see~\cite{KZ19} for  details and other examples.

In~\cite{bT85}, Tsirelson also obtained a probabilistic counterpart of the Steiner formula. In the finite-dimensional case it reads as follows: for any \emph{convex} compact set $K\subset\R^n$,
\begin{align}\label{1021}
\E\,\exp\left(\max_{(t_1,\dots,t_n)\in K}\left[r\sum_{i=1}^{n}t_i{N_i}-\frac{r^2}{2}\sum_{i=1}^{n}t_i^2\right]\right)=\sum_{k=0}^\infty\left(\frac{r}{\sqrt{2\pi}}\right)^kV_k(K).
\end{align}

Further results related to these formulae can be found in the papers of Vitale~\cite{vitale01,vitale08,vitale10} and Chevet~\cite{chevet}.
As was mentioned above, Sudakov and Tsirelson  also obtained the infinite-dimensional versions of~\eqref{2041} and~\eqref{2042}. The basic example in this case is when the index set is the Wiener spiral, and the formula for the intrinsic volumes of its convex hull (found in~\cite{GV01}) leads to a formula for the expected volume of the convex hull of the multidimensional Brownian motion (which was first derived in~\cite{rE14} using a different approach), see~\cite{KZ16} for  details and other examples.

\vskip 10pt

The aim of the present paper is to establish conic versions of~\eqref{2041},~\eqref{2042}, and ~\eqref{1021}.

\section{Notation}
\subsection{Convex cones}
For a set $M\subset\R^n$ denote by $\lin M$ (respectively, $\aff M)$ its linear (respectively, affine) hull, that is, the minimal linear (respectively, affine) subspace containing $M$; equivalently, the set of all linear (respectively, affine) combinations of elements of $M$. The interior and the closure of $M$ will be denoted by $\Int M$ and $\cl M$, respectively.
We write $\relint M$ for the relative interior of $M$ which is the interior of $M$ relative to its affine hull $\aff M$. To avoid problems with measurability,  all sets are tacitly assumed to be Borel.

A non-empty set $C\subset\R^n$ is called  a \emph{convex cone} or just a \textit{cone} if $\lambda_1 x_1+\lambda_2 x_2\in C$ for all $x_1,x_2\in C$ and all $\lambda_1,\lambda_2\geq0$.
For an arbitrary   set  $M\subset \R^n$ let $\pos M$ denote its  conic (or positive) hull defined as the smallest convex cone containing $M$. Equivalently,
\begin{align}
\pos  M
&=
\Big\{\sum_{i=1}^m\lambda_i x_i :x_1,\dots,x_m\in M, \lambda_1,\dots,\lambda_m\geq 0, m\in\N\Big\} \label{2230a}\\
&=
\{\lambda x :x\in \conv M, \lambda\geq 0\}. \label{2230}
\end{align}
The dimension of a convex cone $C$ is defined as the dimension of its linear hull: $\dim C := \dim \lin C$.

\subsection{Grassmann angles}
The \emph{solid angle} of a convex cone $C\subset\R^n$ is defined as
\begin{equation}\label{933}
\alpha(C):=\P[Z\in C],
\end{equation}
where $Z$  is uniformly distributed on the unit sphere in $\R^n$. The maximal possible value of the solid angle in this normalization is $\alpha(\R^n)=1$. If the dimension of $C$ is strictly less than $n$, the solid angle of $C$ is $0$.  If $C\ne\R^n$, then $\P[Z\in C,-Z\in C]=0$ and since the line passing through $Z$ and $-Z$ is equidistributed with  $W_1$ (defined in Section~\ref{subsec:intro}), we obtain that~\eqref{933} is equivalent to
\begin{equation}\label{1304}
\alpha(C) =\frac12\P[C\cap W_1\ne\{0\}].
\end{equation}
In this form, the definition of the solid angle can be generalized as follows.  Fix some $k=0,\dots,n.$ Following Gr\"unbaum~\cite{bG68}  define (with the inverse index order and slightly different notation) the $k$-dimensional \emph{Grassmann angle} of $C$ as the probability for $C$ being intersected by a random $(n-k)$-plane $W_{n-k}$ non-trivially:
\begin{equation}\label{1138}
\gamma_k(C):=\P[C\cap W_{n-k}\ne\{0\}].
\end{equation}
It easily follows that for any convex cone $C\subset\R^n$ with $C\neq \{0\}$,
\[
1 = \gamma_0(C)\geq \gamma_1(C) \geq \ldots  \geq \gamma_n(C)=0.
\]
In the special case when $k=n-1$ we have
\begin{align}\label{739}
  \alpha(C)=\frac12\gamma_{n-1}(C) + \frac12 \ind[C=\R^n].
\end{align}
\begin{remark}\label{1126}
It was shown in~\cite[Eq.~(2.5)]{bG68}) that, as with the intrinsic volumes, the Grassmann angels do not depend on the dimension of the ambient space: if we embed $C$ in $\R^N$ with $N\geq n$, the Grassmann angles  will be the same. In particular, it is convenient to write
\[
\gamma_N(C):=0\quad\text{for all} \quad N\geq\dim C.
\]
\end{remark}
It follows directly from Remark~\ref{1126} that for  a linear $m$-plane  $L_m\subset\R^n$, $m\in \{1,\ldots,n\}$,
\begin{equation}\label{2256}
\gamma_0(L_m) = \ldots = \gamma_{m-1}(L_m) = 1,\quad \gamma_m(L_m) = \ldots = \gamma_n(L_m) = 0.
\end{equation}
For $C=\{0\}$, we have $\gamma_0(C) = \gamma_1(C)= \ldots = 0$.

If $C$ is not a linear subspace, then an equivalent form of~\eqref{1138} is
\begin{align}\label{1026}
    \gamma_k(C)=\P[(\relint C)\cap W_{n-k}\ne\varnothing],
\end{align}
see Corollary~\ref{1027} below.

The Grassmann angles, up to small variations,  coincide with the  \emph{half-tail functionals} defined in~\cite{ALMT14}.

\subsection{Conic intrinsic volumes}

In the forties of the previous century, a conic (also known as spherical) version of the Steiner formula was developed in~\cite{cA48, gH43, lS76}.
In its modern form~\cite{ALMT14, GNP17, MT14}, the formula expresses the size of an angular expansion of a convex cone $C$ in $\R^n$:
\begin{equation}\label{1130}
\P[\dist^2(Z, C)\leq r]=\sum_{k=0}^n\beta_{k,n}(r)\upsilon_k(C),
\end{equation}
where, as above, $Z$ is a random variable uniformly distributed on the $(n-1)$-dimensional unit sphere $\mathbb{S}^{n-1}$, $\dist(Z,C) = \inf\{|Z-y|: y\in C\}$ is the Euclidean distance from $Z$ to $C$,  and $\beta_{k,n}(\cdot)$ is the distribution function of the Beta distribution with parameters $(n-k)/2$ and $n/2$. Since $\beta_{1,n},\dots,\beta_{n,n}$ are linearly independent functions, \eqref{1130} defines the coefficients $\upsilon_k(C)$ uniquely. We call $\upsilon_0(C),\dots,\upsilon_n(C)$ the \emph{conic intrinsic volumes} of $C$ and put $\upsilon_i(C):=0$ for $i>n$. As with the Eucledian intrinsic volumes, the normalization is chosen in such a way  that they do not depend on the dimension of the ambient space:  if we embed $C$ in $\R^N$ with $N\geq n$, the result will be the same.


The $k$th conic intrinsic volume $\upsilon_k(C)$ corresponds to the $(k-1)$st spherical intrinsic volume $\nu_{k-1}(C\cap\s^{n-1})$ in \cite{GHS01, SW08}.

The conic intrinsic volumes satisfy a version of the Gauss--Bonnet theorem (see~\cite[Theorem~6.5.5]{SW08} or~\cite[p.~28]{ALMT14}):
\begin{align}\label{1125}
\upsilon_0(C) + \upsilon_2(C) +\ldots = \frac 12,
\quad
 \upsilon_1(C) + \upsilon_3(C) +\ldots = \frac 12,
\end{align}
provided $C$ is not a linear subspace.
In particular,
\begin{equation}\label{939}
\upsilon_0(C) + \upsilon_1(C) +\ldots +\upsilon_n(C)=1.
\end{equation}
If $L_j\subset\R^n$ is a linear $j$-plane, then
\[
\upsilon_k(L_j)=
\begin{cases}
1, & \text{ if } k=j,\\
0, & \text{ if } k\ne j.
\end{cases}
\]


The Grassmann angles and the conic intrinsic volumes are connected by the  conic version of the Crofton formula (see, e.g.,~\cite[p.~261]{SW08}):
\begin{equation}\label{1818}
\gamma_k(C)=2(\upsilon_{k+1}(C)+\upsilon_{k+3}(C)+\dots),
\end{equation}
for all $k\in \{0,\ldots,n\}$,  where for linear subspaces the factor $2$ should be removed.

\section{Main results}

\subsection{Absorption probabilities for Gaussian cones}
Our first result relates the Grassmann angles of a positive hull of a set with the \emph{absorption probability} of the convex hull of its Gaussian image. As in Section~\ref{subsec:intro}, let $A: \R^{n} \to \R^k$ denote the standard Gaussian random matrix.
\begin{theorem}\label{2054}
For every $k\in \N$ and for an arbitrary  set $M\subset \R^n$ such that $\pos M$ is not a linear subspace,
\begin{equation}\label{1942}
\gamma_{k}(\pos M )=\P[\pos A  M =\R^k]=\P[0\in \Int\conv A M].
\end{equation}
\end{theorem}
\begin{remark}
If $M=C$ is itself a cone (but not a linear subspace), this formula simplifies to
\begin{equation}\label{1942a}
\gamma_{k}(C) = \P[A   C = \R^k] = \P[0\in \Int A C].
\end{equation}
\end{remark}

The reason  we exclude  linear subspaces is  that they are the only cones containing the origin in their relative interior  (see Proposition~\ref{1134} below), which is crucial for the proof of~\eqref{1942}. In the case when $C=L\subset\R^n$ is a linear subspace it follows from~\eqref{2256} and from Proposition~\ref{942} below that instead of~\eqref{1942} we have
\begin{equation*}
\gamma_{k-1}( L )=\P[ A  L =\R^k].
\end{equation*}

The second  part of~\eqref{1942} readily follows from Corollary~\ref{1100}. The main ingredient of the proof of the first part is the spherical invariance of the standard Gaussian distribution.  Namely, we have  that
\begin{align}\label{1050}
    \ker A  \eqdistr W_{n-k}, \qquad k\in \{1,\ldots,n\},
\end{align}
which together with~\eqref{1138} leads to the crucial relation
\begin{equation}\label{2240}
 \gamma_k(C)=\P[\ker A  \cap C\ne\{0\}].
\end{equation}
Applying Lemma~\ref{1146} below we obtain~\eqref{1942}. The details are postponed to Section~\ref{2248}.

The next corollary, which can be considered as a conic version  of Sudakov's formula~\eqref{2041}, connects the $0$th conic intrinsic volume with the \emph{persistence probability} of the isonormal Gaussian process.
\begin{corollary}\label{2048}
For an arbitrary set $M\subset \R^n$ such that $\pos M$ is not a linear subspace,
\begin{equation}\label{947}
\upsilon_0(\pos M )=\P\Big[\inf_{(t_1,\dots,t_n)\in M}\sum_{i=1}^{n}t_i{N_i}\geq0\Big].
\end{equation}
\end{corollary}
\begin{proof}
Let $C= \pos M$. By~\eqref{1125} and~\eqref{1818} we have
$$
\upsilon_0(C) = \frac 12 -\upsilon_2(C) -  \upsilon_4(C)-\ldots = \frac 12(1- \gamma_1(C)).
$$
Taking $k=1$ in Theorem~\ref{2054}, we obtain
$$
\upsilon_0(C) = \frac 12 (1 - \P[\pos A M =\R^1]) = \P[A M \subset [0,\infty)]
$$
because the conic hull of $A M$ is not $\R^1$ if and only if  $A M$ is either completely contained in $[0,\infty)$ or in $(-\infty,0]$, the probabilities of both possibilities being equal. Recalling that $A$ is the matrix $(N_1,\ldots,N_n)$ completes the proof.
\end{proof}

\begin{remark}
Another way to see~\eqref{947} is to use the fact that $\upsilon_0(C) = \upsilon_n(C^{\circ}) = \alpha(C^{\circ})$, where $C^{\circ} = \{x\in \R^n:\langle x, y\rangle \leq 0 \text{ for all }y\in C\}$ is the polar cone of $C$. The angle of $C^\circ$ is the probability that $\langle N, y\rangle\leq 0$ for all $y\in C$, where $N$ is a random vector in $\R^n$ with any rotationally invariant distribution. For example, we may take $N=(-N_1,\ldots,-N_n)$ to be standard Gaussian, leading to~\eqref{947}.
\end{remark}

\subsection{Grassmann angles of Gaussian cones}
In the next theorem, we compute the expected Grassmann angles of the Gaussian image of a convex cone. We recall that $A:\R^n\to\R^k$ denotes a standard Gaussian matrix, where $k\in\N$.
\begin{theorem}\label{655}
Let $C\subset\R^n$ be a convex cone.    Then, the dimension of $AC$ is $m := \min(\dim C, k)$ with probability $1$, and for all  $j \in \{0,\ldots,m-1\}$ we have
\begin{equation*}
\E [\gamma_j(A  C)]
=
\gamma_j(C).
\end{equation*}
\end{theorem}
The proof is postponed to Section~\ref{2248}. After the first version of this paper has been uploaded to the arXiv, it has been pointed to us by Martin Lotz that Theorem~\ref{655} follows from Lemma IV.9 in~\cite{amelunxen_lotz_walvin} (which was communicated to the authors of~\cite{amelunxen_lotz_walvin} by M.\ McCoy) by combining it with the polar decomposition of the Gaussian matrix $A$.

Let us state several corollaries of Theorems~\ref{2054} and~\ref{655}.
\begin{corollary}\label{1151}
Let $C\subset \R^n$ be a convex cone which is  not a linear subspace. Then, for all $k\in \{1,\ldots,\dim C\}$ and $j \in \{0,\ldots,k-1\}$,
\begin{equation}\label{eq:E_gamma_j_AC_ind}
\E\left[\gamma_j(A  C)\ind[A  C\ne \R^k]\right]
=
\gamma_j(C) - \gamma_k(C).
\end{equation}
If, additionally, $\gamma_k(C)\neq 1$, then
\begin{equation}\label{eq:E_gamma_j_AC_cond}
\E[\gamma_j(A  C)\,|\,AC\ne \R^k]= \frac{\gamma_j(C)-\gamma_k(C)}{1-\gamma_k(C)}.
\end{equation}
\end{corollary}
\begin{proof}
By Theorem~\ref{655} we have
\begin{align*}
\E\left[\gamma_j(A  C)\ind[A  C \ne \R^k]\right]
&=
\E\left[\gamma_j(A  C)\right] - \E\left[\gamma_j(A  C)\ind[A  C = \R^k]\right]\\
&=
\gamma_j(C) - \P[A  C = \R^k]
\end{align*}
since on the event $A  C = \R^k$, we have $\gamma_j(A  C) = 1$ for all $j<k$. To prove~\eqref{eq:E_gamma_j_AC_ind}, recall that $\P[A  C = \R^k] = \gamma_k(C)$ by Theorem~\ref{2054}. To prove~\eqref{eq:E_gamma_j_AC_cond}, use~\eqref{eq:E_gamma_j_AC_ind} and the formula  $\P[A  C \neq  \R^k] = 1-\gamma_k(C)$. To avoid division by $0$, we have to exclude the case $\gamma_k(C) = 1$ (which occurs if $C$ is a half-space).
\end{proof}

The following corollary  gives a formula for the expected solid angle of $AC$.
\begin{corollary}\label{cor:E_alpha C}
Let $C\subset \R^n$ be a convex cone which is not a linear subspace.  For all $k\in \{1,\ldots,\dim C\}$ we have
\begin{align*}
\E [\alpha(AC)]
&=
\frac{\gamma_k(C) + \gamma_{k-1}(C)}2 \\
&=
\upsilon_{k}(C)+\upsilon_{k+1}(C)+ \upsilon_{k+2}(C)+\dots,\\
\E [\alpha(AC) \ind[A  C\neq \R^k]]
&=
\frac{\gamma_{k-1}(C) - \gamma_{k}(C)}2 \\
&=
\upsilon_{k}(C) - \upsilon_{k+1}(C) + \upsilon_{k+2}(C) -\dots.
\end{align*}
\end{corollary}
\begin{proof}
In both cases, the second equality follows from the conic Crofton formula~\eqref{1818}.  To prove the first equalities, observe that by~\eqref{739},
\begin{align*}
\alpha(A  C)=\frac12\gamma_{k-1}(A  C)+\frac12\ind[A  C=\R^k].
\end{align*}
Multiplying by the indicator function, we obtain
$$
\alpha(A  C) \ind[A  C\neq \R^k] = \frac12 \gamma_{k-1}(A  C)\ind[A  C\neq \R^k].
$$
Taking the expectation in the above two equalities, applying Theorem~\ref{655} and Corollary~\ref{1151} with $j=k-1$ and noting that $\P[A C\ne \R^k]=1-\gamma_k(C)$ by Theorem~\ref{2054} yields the required formula.
\end{proof}

\begin{corollary}
Let $C\subset \R^n$ be a convex cone and $k=n$. Then,
$$
\E [\alpha(AC)] = \alpha (C).
$$
\end{corollary}
\begin{proof}
The statement is trivial if $\dim C < n$ or $C=\R^n$. Otherwise, we can apply Corollary~\ref{cor:E_alpha C} with $k=n$ and use that $\gamma_n(C) = 0$ and $\alpha(C) = \frac 12 \gamma_{n-1}(C)$.
\end{proof}

\begin{remark}
In Theorems~\ref{2054} and~\ref{655}, as well as in the corollaries of Theorem~\ref{655}, it is possible to replace the Gaussian matrix $A$ by the orthogonal projection onto a uniformly distributed random linear subspace $W_{k}$, $k\in \{1,\ldots,n\}$. In fact, all proofs apply with minor changes since the essential property of $A$ used there was the fact that $\ker A$ has the same distribution as $W_{n-k}$, which is true for the orthogonal projection on $W_{k}$ as well.
For example, the analogue of the first equation in Corollary~\ref{cor:E_alpha C} is
$$
\E [\alpha(C|W_{k})]
=
\frac{\gamma_k(C) + \gamma_{k-1}(C)}2.
$$
This formula was obtained by Glasauer~\cite{glasauer_phd}; see~\cite[p.~263]{SW08} and also~\cite[Lemma~5.1]{convex_hull_sphere} for the conic version stated here. In particular, we have $\E [\alpha(AC)] = \E [\alpha(C|W_{k})]$, which reminds of the theorem of Baryshnikov and Vitale~\cite{baryshnikov_vitale} stated in the introduction; see~\eqref{eq:bar_vit_2}.
\end{remark}

\subsection{A probabilistic conic Steiner formula}
Our next result establishes a probabilistic version of the conic Steiner formula and a conic version of Tsirelson's formula~\eqref{1021}.
\begin{theorem}\label{1338}
Let $N=(N_1,\ldots,N_n)$ be a standard Gaussian vector in $\R^n$ and let $M\subset\R^n$ be an arbitrary set. Then, for every $r>0$,
$$
\E\exp\left(\frac{1-r^{-2}}{2}\sup_{x\in (\conv M) \backslash \{0\}}\frac{\langle N,x\rangle_+^2}{|x|^2}\right)=\sum_{k=0}^n r^k\upsilon_k(\pos M).
$$
\end{theorem}
In a different form, this relation was first obtained in~\cite[Remark~4.7]{MT14} by McCoy and Tropp. In Section~\ref{2225} we will prove the equivalence of their result and Theorem~\ref{1338}.
\vskip 10pt

In the next section, we give several applications of Theorems~\ref{2054} and~\ref{655}.

\section{Examples}
\subsection{Example I: Absorption probability for Gaussian polytopes}
Let $M$ be the standard orthonormal basis in $\R^n$:
\begin{align*}
    M=\{e_1,\dots,e_n\}.
\end{align*}
Then $\conv M$ is the regular $(n-1)$-dimensional simplex and $\pos M$ is the non-negative orthant $[0,\infty)^n$.

Since $X_1 := A e_1,\dots, X_n := A e_n$ are independent standard Gaussian vectors in $\R^k$, we obtain that
\begin{align*}
    \cP_n:=\conv A M = \conv (X_1,\ldots,X_n)
\end{align*}
is a \emph{Gaussian polytope} in $\R^k$.
Theorem~\ref{2054} relates the absorption probability for the Gaussian polytope to the Grassmann angles of the non-negative orthant.
The conic intrinsic volumes of the non-negative orthant are well known  (see, e.g.,~\cite{ALMT14}):
\begin{align*}
\upsilon_k(\pos M)=\frac{\binom nk}{2^n}, \quad k=0,\ldots,n.
\end{align*}
It follows from~\eqref{1818} that the Grassmann angles of the non-negative orthant are given by
$$
\gamma_k(\pos M)
=
\frac{1}{2^{n-1}} \left(\binom{n}{k+1} + \binom{n}{k+3}+\ldots \right)
=
\frac{1}{2^{n-1}}\sum_{i=k}^{n-1}\binom{n-1}{i},
$$
for $k\in \{0,\ldots,n-1\}$.
Theorem~\ref{2054} yields the following formula for the probability that the Gaussian polytope absorbs the origin:
\begin{align*}
\P[0 \in \Int \cP_n]
=
\gamma_k(\pos M)
=
\frac{1}{2^{n-1}}\sum_{i=k}^{n-1}\binom{n-1}{i}.
\end{align*}
This recovers a special case of Wendel's formula~\cite{jW62}.

\subsection{Example II: Angles of Gaussian simplices}
Given a polytope $P$, denote by $\cF_\ell(P)$ the set of its $\ell$-dimensional faces, $\ell\in \{0,\ldots,\dim P\}$. The \textit{tangent cone} of $P$ at its face $F\in \cF_\ell(P)$ is defined as
$$
T(F,P) := \{y\in \R^d\colon  \exists \varepsilon>0 \text{ such that } f_0 + \varepsilon y \in P\}
$$
where $f_0$ is any point in $\relint F$. We are interested in the sum of the $j$-th Grassmann angles at all $\ell$-dimensional faces of $P$ denoted by
$$
S_{\ell, j}(P) := \sum_{F\in \cF_\ell(P)} \gamma_j(T(F,P)).
$$
The following theorem states that, on average, the sum of Grassmann angles of a Gaussian simplex equals the sum of Grassmann angles of the regular simplex.
\begin{theorem}
Consider the Gaussian simplex $\cP_n:= \conv(X_1,\ldots,X_n)$, where $X_1,\ldots,X_n$ are independent standard Gaussian random vectors in $\R^k$ and $2\leq n\leq k+1$. Let also $\Delta_n:= \conv (e_1,\ldots,e_n)\subset \R^n$ be a regular simplex with $n$ vertices. Then, for all $\ell\in \{0,\ldots,n-2\}$ and $j\in \{0,\ldots, n-2\}$ we have
$$
\E S_{\ell,j}(\cP_n) =  S_{\ell,j}(\Delta_n).
$$
\end{theorem}
A special case of this result when $n=k+1$ (the full-dimensional simplex), $\ell = 0$ (angles are considered at vertices),  and $j = k-1$ has been obtained in~\cite{KZ18}.

\begin{proof}
Let $F_{\ell}:=\conv(X_1,\ldots,X_{\ell+1})$ and $G_{\ell}:= \conv(e_1,\ldots,e_{\ell+1})$. By exchangeability, it suffices to show that
\begin{equation}\label{eq:E_gamma_j_gamma_j}
\E \gamma_j(T(F_{\ell},\cP_n))
=
\gamma_j(T(G_{\ell},\Delta_n)).
\end{equation}
The tangent cones of $\cP_n$ and $\Delta_n$ at $F_\ell$ and $G_\ell$ are given by
\begin{align*}
T(F_{\ell},\cP_n)
&=
\pos (X_i -\bar X_{\ell+1}: i=1,\ldots,n),\\
T(G_{\ell},\Delta_n)
&=
\pos (e_i -\bar e_{\ell+1}: i=1,\ldots,n),
\end{align*}
where $\bar X_{\ell+1} = \frac 1 {\ell+1} (X_1+\ldots+X_{\ell+1})$ and $\bar e_{\ell+1} = \frac 1 {\ell+1} (e_1 + \ldots + e_{\ell+1})$. If $A:\R^n\to\R^k$ is a Gaussian random matrix, then we can identify $X_1:= Ae_1,\ldots, X_n:= A e_n$, so that
$$
T(F_{\ell},\cP_n) = A(T(G_{\ell},\Delta_n)).
$$
Applying Theorem~\ref{655}, we obtain~\eqref{eq:E_gamma_j_gamma_j}.
\end{proof}

By essentially the same method, it is possible to compute $\E S_{\ell,j}(\cP_n)$ in the case when $n$ need not satisfy $n \leq k+1$.

\subsection{Example III: Convex hulls of Gaussian random walks}
Let $M$ be the set of partial sums of the standard orthonormal vectors:
\begin{align*}
    M=\{e_1,e_1+e_2,\dots,e_1+\dots+e_n\}\subset \R^n.
\end{align*}
Then,  $\pos M$ is the Weyl chamber of type $B_n$:
\begin{align*}
\pos M = \{(t_1,\ldots,t_n)\in\R^n: t_1\geq t_2\geq\ldots \geq t_n\geq0\}.
\end{align*}
Since $X_1 := A e_1,\dots, X_n := A e_n$ are independent standard Gaussian vectors in $\R^k$, we obtain that
\begin{align*}
    \cQ_n:=\conv A M = \conv(X_1,X_1+X_2,\ldots,X_1+\ldots+X_n)
\end{align*}
is the \emph{convex hull of the Gaussian random walk} in $\R^k$.
The conic intrinsic volumes of the Weyl chamber were first computed in~\cite{KS11}, see also~\cite{KVZ17a}:
\begin{align*}
\upsilon_k(\pos M) = \frac {B(n,k)}{2^n n!}, \quad k=0,\ldots,n,
\end{align*}
where $B(n,k)$ is the $k$th coefficient of the polynomial
$$
(t+1)(t+3)\ldots (t+2n-1) = \sum_{k=0}^n B(n,k) t^k.
$$
Crofton's formula~\eqref{1818} yields the Grassmann angles of the Weyl chamber of type $B_n$:
$$
\gamma_k(\pos M) = \frac{2}{2^n n!} (B(n,k+1) + B(n,k+3) + \ldots),
$$
for $k\in \{0,\ldots,n-1\}$. Theorem~\ref{2054} allows to compute the absorption probability for the Gaussian random walk:
\begin{align*}
\P[0 \in \Int \cQ_n] =  \gamma_k(\pos M) = \frac{2}{2^n n!} (B(n,k+1) + B(n,k+3) + \ldots).
\end{align*}
This recovers a special case of Theorem~1.1 in~\cite{KVZ17a} (where the increments need are not required to be Gaussian).

\vskip 10pt
The rest of the paper is devoted to the proof of Theorems~\ref{2054},~\ref{655}, and~\ref{1338}. We start with the section where we provide some essential facts about convex cones.

\section{Properties of convex cones}\label{1052}
An essential step in proving Theorems~\ref{2054} and~\ref{655} is the following property of  convex hulls.
\begin{lemma}\label{1146}
 Let $k\in \{1,\ldots,n-1\}$. Consider some  convex cone  $C\subset \R^n$ and some (deterministic) matrix $A\in\R^{k\times n}$. The following two conditions are equivalent:
\begin{enumerate}
    \item[(a)] $(\relint C)\cap \ker A\ne\varnothing$;
    \item[(b)] $AC=\lin AC$ (that is, $AC$ is a linear subspace).
\end{enumerate}
\end{lemma}
The proof is postponed to Section~\ref{705}.
In the remaining part of this section, we collect some basic facts about convex cones.  Most of them are well known, but for the reader's convenience we provide their proofs.

\vskip 10pt

The first property gives a criterion for the conic hull of a set to be a linear subspace.
\begin{proposition}\label{1134}
For  any   set  $M\subset \R^n$ the following two conditions are equivalent:
\begin{enumerate}
    \item[(a)] $0\in \relint \conv M$;
    \item[(b)] $\pos M=\lin M$.
\end{enumerate}
\end{proposition}
\begin{proof}
It is obvious that (a) implies (b). Now assume that $\pos M=\lin M$. Let $n':=\dim \lin M$ and let $e_1,\dots,e_{n'}$ be an orthonormal basis in $\lin M$. By~(b) and~\eqref{2230} we have that there exist $a_1,b_1\dots,a_{n'},b_{n'}>0$ such that
\[
a_1e_1,-b_1e_1,\dots,a_{n'}e_{n'},-b_{n'}e_{n'}\in \conv M,
\]
and  (a) follows.
\end{proof}
In the case when $\lin M = \R^n$ the proposition simplifies as follows.
\begin{corollary}\label{1100}
For  any   set  $M\subset \R^n$ the following two conditions are equivalent:
\begin{enumerate}
    \item[(a)] $0\in\Int \conv M$;
    \item[(b)] $\pos M=\R^n$.
\end{enumerate}
\end{corollary}

The next property  follows directly from the definition of the conic hull given in~\eqref{2230a}.
\begin{proposition}\label{1054}
Let $k\in\N$. Consider  some matrix $A\in\R^{k\times n}$. Then for any set $M\subset \R^n$,
 \[
 \pos AM=A\pos M.
 \]
\end{proposition}

It is easy to see that if $C\subset\R^n$ is a convex cone, then $(\relint C)\cup\{0\}$ is a convex cone, too. It is to be expected that this cone has the same Grassmann angles as $C$.
\begin{proposition}\label{801}
For any convex cone $C\subset\R^n$ and $k\in \{0,1,\dots,n\}$,
\[
\gamma_k((\relint C) \cup\{0\})=\gamma_k(C).
\]
\end{proposition}
\begin{proof}
Let $x\in\R^n$. For any set $M\subset \R^n$ we have $\dist(x,M)=\dist(x,\cl M)$ and it is well known  (see, e.g., \cite[Theorem~1.1.14]{SchneiderBook}) that for convex $M$,
\begin{equation*}
\cl M = \cl \relint M.
\end{equation*}
Therefore,  for any convex cone $C\subset\R^n$ and any $x\in\R^n$,
\[
\dist(x,C)=\dist(x,\relint C)=\dist(x,(\relint C)\cup\{0\}).
\]
Applying~\eqref{1130} and~\eqref{1818} completes the proof.
\end{proof}
\begin{corollary}\label{1027}
    If $C\subset \R^n$ is  a convex cone, then for all $k\in \{0,1,\dots,n-1\}$,
\begin{align*}\label{1026}
    \gamma_k(C)=\P[(\relint C)\cap W_{n-k}\ne\varnothing].
\end{align*}
In particular, the event $\{C\cap W_{n-k}\neq \{0\}\}\cap \{(\relint C)\cap W_{n-k} = \varnothing\}$ has probability $0$.
 \end{corollary}
\begin{proof}
 The statement is trivial if $C$ is a linear subspace, see~\eqref{2256}, hence we exclude this possibility in the following.     By Proposition~\ref{801} and by~\eqref{1138},
    \begin{align*}
       \gamma_k(C)=\gamma_k((\relint C)\cup \{0\})&=\P[((\relint C)\cup \{0\})\cap W_{n-k}\ne \{0\}].
    \end{align*}
Recall that $C\ne\lin C$. Therefore it follows from Proposition~\ref{1134}  that $ 0\not\in\relint C$, and thus we have
    \begin{align*}
        \gamma_k(C)=\P[(\relint C)\cap W_{n-k}\ne \varnothing].
    \end{align*}
    The second claim of the corollary follows  from this relation after recalling~\eqref{1138}.
\end{proof}

Recall that  $A\in\R^{k\times n}$  denotes the standard Gaussian random matrix.
\begin{proposition}\label{942}
For any $k \in \N$ and for arbitrary cone $C\subset\R^n$,
\[
\P[\dim A C=\min(k,\dim C)]=1.
\]
\end{proposition}
\begin{proof}
Put $m:=\min(k,\dim C)$. We obviously have $\dim AC\leq m$. It remains to show that the event $\dim AC \leq m-1$ has probability zero.
There exist linearly independent vectors $v_1,\dots, v_{m}\in C$. Let
\[
L_m:=\lin(v_1,\dots,v_m).
\]
If $\dim A C\leq m-1$, then $Av_1,\dots, A v_{m}$ are linearly dependent, that is, there exists $(c_1,\dots,c_m)\in\R^m\setminus\{0\}$ such that
\[
\sum_{i=1}^mc_iAv_i=A\sum_{i=1}^mc_iv_i=0,
\]
which is equivalent to
\[
L_m\cap\ker A\ne\{0\}.
\]
Thus,
\begin{align*}
  \P[\dim A C\leq m-1]\leq  \P[L_m\cap\ker A\ne\{0\}].
\end{align*}
If $k\geq n$, then $\ker A= \{0\}$ with probability $1$, hence the right-hand side is $0$. Let $k<n$.  Recalling~\eqref{1050},~\eqref{1138}, and~\eqref{2256} we arrive at
\begin{align*}
  \P[\dim A C\leq m-1]&\leq   \P[L_m\cap  W_{n-k}\ne\{0\}]=\gamma_k(L_m)=0
\end{align*}
since $k\geq m$.
\end{proof}


\section{Proofs of Theorem~\ref{2054} and Theorem~\ref{655}}\label{2248}

\begin{proof}[Proof of Theorem~\ref{2054}]
As was mentioned above, the second part of~\eqref{1942} follows directly from Corollary~\ref{1100}. Let us prove the first one.   Due to Proposition~\ref{1054} the task is to show that
\begin{equation}\label{1049}
\gamma_k(C)=\P[AC=\R^k],
\end{equation}
where $C:=\pos M$.

By assumption, $C\ne\lin C$. We may assume that $k\in \{1,\ldots,n-1\}$ since for $k\geq n$ both sides of~\eqref{1049} vanish.  It follows  from Corollary~\ref{1027} that
\begin{align*}
    \gamma_k(C)=\P[(\relint C)\cap W_{n-k}\ne \varnothing].
\end{align*}
Applying Lemma~\ref{1146} together with~\eqref{1050} gives
\begin{align}\label{1127}
    \gamma_k(C)=\P[A   C=\lin A   C].
\end{align}
If $\dim C \geq k$, then by Proposition~\ref{942} we have $\P[\dim A   C = k]=1$, hence $\P[\lin A C = \R^k] = 1$ and~\eqref{1049} follows. If $\dim C \leq k$, then again by Proposition~\ref{942} and by~\eqref{2256} both sides of~\eqref{1049} vanish.
\end{proof}

\begin{proof}[Proof of Theorem~\ref{655}]
We shall frequently use the following invariance property of the Gaussian random matrix: if $O_1:\R^n \to\R^n$ and $O_2:\R^k\to\R^k$ are (deterministic) orthogonal transformations, then the random matrix $O_2AO_1$ has the same distribution as $A$.

By Proposition~\ref{942}, the dimension of $AC$ is $m := \min(\dim C, k)$ with probability $1$. Let us prove that for all $j \in \{0,\ldots,m-1\}$,
\begin{equation}\label{eq:E_gamma_j_proof}
\E [\gamma_j(A  C)]
=
\gamma_j(C).
\end{equation}

Without restriction of generality, we may assume that $\dim C = n$. Indeed, if $\ell := \dim C < n$, the invariance property of $A$ shows that we can assume $\lin C= \lin (e_1,\ldots, e_\ell)$, where $e_1,\ldots,e_n$ is the standard orthonormal basis of $\R^n$. Identifying $\lin C$ with $\R^\ell$ and noting that the restriction of $A$ to $\R^\ell$ is also  a Gaussian linear operator from $\R^\ell$ to $\R^k$, we are in the setting when $C$ has full dimension.

So, let $\dim C = n$. The dimension of $AC$ is then $m= \min (k, n)$, with probability $1$. By Corollary~\ref{1027} applied to the cone $AC\subset \R^k$, for all $j\in \{0,1,\ldots, k-1\}$ we have
$$
\E [\gamma_j(AC)] = \P [U_{k-j} \cap \relint (AC)\neq \varnothing],
$$
where $U_{k-j}$ is a random, uniformly distributed, $k-j$-dimensional linear subspace of $\R^k$ that is independent of $A$. By the invariance property of $A$,  the random cone $AC$ has rotationally invariant distribution, that is it has the same distribution as $O_2 A C$ for every (deterministic) orthogonal transformation $O_2:\R^k\to\R^k$. It follows that we can replace $U_{k-j}$ by any \textit{deterministic} linear subspace $L_{k-j}\subset \R^k$ of dimension $k-j$:
\begin{equation}\label{eq:proof1}
\E [\gamma_j(AC)] = \P [L_{k-j} \cap \relint (AC)\neq \varnothing].
\end{equation}
In a moment, we shall show that  this implies that
\begin{equation}\label{eq:proof2}
\E [\gamma_j(AC)] = \P [A^{-1}L_{k-j} \cap \Int C\neq \varnothing],
\end{equation}
where $A^{-1}L_{k-j}$ denotes the preimage of $L_{k-j}$ under $A$.  Given this, we can complete the proof of~\eqref{eq:E_gamma_j_proof} as follows. Recall that $j < \min (k,n)$.  With probability $1$, the matrix $A$ has full rank and the linear subspace $A^{-1} L_{k-j}$ has codimension  $j$ in $\R^n$. By the invariance property of $A$, the distribution of $A^{-1} L_{k-j}$ is invariant with respect to orthogonal transformations of $\R^n$. Hence, $A^{-1} L_{k-j}$ has the same distribution as $W_{n-j}$ and we have
$$
\E [\gamma_j(AC)] = \P [W_{n-j} \cap \Int C\neq \varnothing] = \gamma_j(C)
$$
by Corollary~\ref{1027}.

To complete the proof of~\eqref{eq:E_gamma_j_proof} it remains to check the equivalence of~\eqref{eq:proof1} and~\eqref{eq:proof2}, for which it suffices to verify that
\begin{equation}\label{eq:proof_tech1}
\{A^{-1}L_{k-j} \cap \Int C \neq \varnothing\} \subset \{L_{k-j} \cap \relint (AC)\neq \varnothing\}
\end{equation}
and
\begin{equation}\label{eq:proof_tech2}
\P[\{L_{k-j} \cap \relint (AC)\neq \varnothing\} \backslash \{A^{-1}L_{k-j} \cap \Int C \neq \varnothing\}] = 0.
\end{equation}

Let us prove~\eqref{eq:proof_tech1}.  If the event $\{A^{-1}L_{k-j} \cap \Int C \neq \varnothing\}$ occurs, then there exists $v\in \Int C$ such that $Av \in L_{k-j}$. Let $O\subset C \subset \lin C=\R^n$ be an open set such that $v\in O$. Consider $A$ as a linear operator defined on $\lin C = \R^n$ with the image $A \lin C = \lin (AC)$.
Since linear surjective maps are open (that is, they map open sets to open sets), the set $AO$ is relatively open in $\lin(AC)$. Since $Av \in AO \subset AC$, it follows that $Av \in \relint (AC)$. Recalling that $Av\in L_{k-j}$, we conclude that the event $\{L_{k-j} \cap \relint (AC)\neq \varnothing\}$ occurs, thus proving~\eqref{eq:proof_tech1}.

Let us prove~\eqref{eq:proof_tech2}. Consider the event $E:= \{L_{k-j} \cap \relint (AC)\neq \varnothing\} \backslash \{A^{-1}L_{k-j} \cap \Int C \neq \varnothing\}$. On this event, we may take some $y\in L_{k-j}\cap \relint (AC)$. We may assume that $y\neq 0$. Indeed, if it happens that $y=0$, then $0\in \relint (AC)$, which implies that $AC = \lin (AC)$ is a linear subspace of $\R^k$ by Proposition~\ref{1134} and, since
$$
\dim \lin (AC) + \dim L_{k-j} = \min (k,n) + k - j > k,
$$
we have $L_{k-j}\cap \relint (AC)  = L_{k-j}\cap \lin (AC) \neq \{0\}$, which means that we may take $y\neq 0$. Now, recalling that $y\in L_{k-j}\cap \relint (AC)$ it follows that there exists $x\in C$  such that $y = Ax$ and $x\neq 0$. Then,  $x\in A^{-1}L_{k-j}$ and it follows that on the event $E$, we have $A^{-1}L_{k-j}\cap C \neq \{0\}$. To summarize,
$$
E\subset \{A^{-1}L_{k-j} \cap C \neq \{0\}\} \cap \{A^{-1}L_{k-j} \cap  \Int C  = \varnothing\}.
$$
However, the event on the right hand-side has probability $0$ because, as we explained above, $A^{-1}L_{k-j}$ is a random linear subspace of $\R^n$ with the same distribution as $W_{n-j}$, and hence the claim follows from Corollary~\ref{1027}.  The proof of~\eqref{eq:proof_tech2} is complete.
\end{proof}

\section{Proof of Theorem~\ref{1338}}\label{2225}
It was shown in~\cite[Remark~4.7]{MT14} that for any $r>0$ and any convex cone $C\subset\R^n$,
\begin{equation}\label{1247}
\E\,\exp\left(\frac{1-r^{-2}}{2}|\Pi_{C}(N)|^2\right)=\sum_{k=0}^n r^{k}\upsilon_k(C),
\end{equation}
where $\Pi_{C}\,:\,\R^n\to C$ is a metric projection on the cone $C$ defined as
$$
\Pi_{C}(x):=\arg\min\{|x-y|:y\in \cl {C}\}.
$$
For any $u\in\mathbb S^{n-1}$ it obviously holds
$$
\argmin\limits_{\lambda u}\{|x-\lambda u|:\lambda\geq0\}=\langle x,u\rangle_+ u.
$$
By convention, the left-hand side is defined to be $\lambda u$, where $\lambda \geq 0$ is chosen to minimize the function $\lambda \mapsto |x-\lambda u|$. Similar convention will be used below.
Therefore,
\begin{align*}
    \Pi_{C}(x)&=\argmin\limits_{\lambda u}\{|x-\lambda u|:u\in \cl {C}\cap\mathbb S^{n-1},\lambda\geq 0\}
    \\&=\argmin\limits_{\langle x,u\rangle_+u}\{|x-\langle x,u\rangle_+u|:u\in \cl {C}\cap\mathbb S^{n-1}\}.
\end{align*}
By the Pythagorean theorem,
\begin{align*}
    |x|^2=|\langle x,u\rangle u|^2+|x-\langle x,u\rangle u|^2=|\langle x,u\rangle|^2+|x-\langle x,u\rangle u|^2.
\end{align*}
Thus,
\begin{align*}
    \Pi_{C}(x)=\argmax\limits_{\langle x,u\rangle_+u}\{\langle x,u\rangle_+:u\in \cl {C}\cap\mathbb S^{n-1}\}.
\end{align*}
Taking the norm, we obtain
\begin{align*}
    |\Pi_{C}(x)|&=\max\{\langle x,u\rangle_+:u\in \cl {C}\cap\mathbb S^{n-1}\}
    \\&=\sup\{\langle x,u\rangle_+:u\in {C}\cap\mathbb S^{n-1}\}
    \\&=\sup\{{\langle x,z\rangle_+}/{|z|}:z\in {C}\backslash\{0\}\}.
\end{align*}
Taking $C=\pos M$ leads to
\begin{align*}
    |\Pi_{\pos M}(x)|&=\sup\{\langle x,z\rangle_+/|z|:z\in (\pos M)\backslash\{0\}\}
    \\&=\sup\{\langle x,z\rangle_+/|z|:z\in (\conv M)\backslash\{0\}\},
\end{align*}
which together with~\eqref{1247} completes the proof. \hfill $\Box$

\section{Proof of Lemma~\ref{1146}}\label{705}

The crucial step in the proof of Lemma~\ref{1146} is the next fundamental result which gives a criterion for convex sets to be properly separated.
Let $M_1$ and $M_2$ be non-empty convex sets in $\R^n$. An affine hyperplane $H$ is said to separate them properly, if they lie in the different closed half-spaces generated by $H$ and their union is not contained in $H$.

\begin{theorem}[Separation theorem]\label{652}
The relative interiors of two non-empty convex sets $M_1, M_2\subset\R^n$ do not intersect if and only if there exists an affine hyperplane which separates them properly.
\end{theorem}
\begin{proof}
See, e.g.,~\cite[Theorem~11.3]{rR70} or~\cite[Theorem~1.3.8]{SchneiderBook}.
\end{proof}

We are now ready to prove Lemma~\ref{1146}. Let $n':=\dim C$.
Consider an orthonormal basis $e_1,\dots,e_n$ in $\R^n$ such that
\[
e_1,\dots,e_{n'}\in \lin C.
\]
We prove that (a) implies (b). It follows from~(a) that  for some $r>0$ and $t\in \ker A$,
\[
t+\conv(\pm re_1,\dots, \pm re_{n'})\subset C.
\]
Applying $A$ gives
\[
\conv(\pm rAe_{1},\dots, \pm rAe_{n'})\subset AC,
\]
which is equivalent to
\[
\pos(\pm Ae_{1},\dots, \pm Ae_{n'})\subset AC.
\]
Since $e_1,\dots,e_{n'}$  is a basis of $\lin C$, it follows that the linear hull of $Ae_{1},\dots, Ae_{n'}$ is $A \lin C$. Hence,
$$
\pos(\pm Ae_{1},\dots, \pm Ae_{n'}) = \lin(Ae_{1},\dots, Ae_{n'}) = A \lin C = \lin (AC).
$$
Altogether, it follows that $\lin (AC)= AC$, thus (b) holds.

To prove that (b) implies (a), we assume, by contraposition, that $(\relint C)\cap \ker A=\varnothing$. Then it follows from Theorem~\ref{652} that there exists an affine hyperplane $H$ which separates $C$ and $\ker A$ properly. Since $\ker A$ is a linear subspace and $0\in C$, from the definition of  proper separation it readily follows that $\ker A\subset H$. In particular, $H$ is linear, that is, it passes through the origin.

Let $H^+$ be the unique closed subspace whose boundary is $H$ and which contains $C$. Further, let $u$ be the normal vector orthogonal to $H$ and lying in $H^+$. For all $x\in C$ we have
\[
\langle x,u\rangle\geq0.
\]
Moreover, by the definition of  proper separation, $C$ is not a subset of $H$, thus there exists $x_0$ such that
\begin{equation}\label{811}
 x_0\in C\quad\text{and}\quad   \langle x_0,u\rangle>0.
\end{equation}
Since $\image A^\top$ is the orthogonal complement of $\ker A$ and thus contains $u$, there is a vector $v\in\R^k$ such that $A^\top v=u$. Note that $v$ is non-zero. For all $x\in C$,
\[
\langle Ax,v\rangle=\langle x,A^\top v\rangle=\langle x,u\rangle\geq0,
\]
which means that $AC$ is contained in the following half-space:
\[
AC\subset\{y\in\R^k:\langle y,v\rangle\geq0\}.
\]
On the other hand, recalling~\eqref{811} gives
\[
\langle -Ax_0,v\rangle=-\langle x_0,u\rangle<0,
\]
which means that
\[
\lin AC\not\subset\{y\in\R^k:\langle y,v\rangle\geq0\}.
\]
Thus $AC\ne \lin AC$. \hfill $\Box$

\subsection*{Acknowledgement} We are grateful to M.\ Lotz for pointing out Reference~\cite{amelunxen_lotz_walvin}. 

\bibliographystyle{plain}
\bibliography{bib1}

\begin{thebibliography}{10}

\bibitem{cA48}
C.~B. Allendoerfer.
\newblock Steiner's formulae on a general {$S^{n+1}$}.
\newblock {\em Bull. Amer. Math. Soc.}, 54:128--135, 1948.

\bibitem{ALMT14}
D.~Amelunxen, M.~Lotz, M.~B. McCoy, and J.~A. Tropp.
\newblock Living on the edge: phase transitions in convex programs with random
  data.
\newblock {\em Inf. Inference}, 3(3):224--294, 2014.

\bibitem{amelunxen_lotz_walvin}
D.~Amelunxen, M.~Lotz, and J.~Walvin.
\newblock {Effective condition number bounds for convex regularization.}
\newblock Preprint at arXiv:1707.01775, 2017.

\bibitem{baryshnikov_vitale}
Y.~M. {Baryshnikov} and R.~A. {Vitale}.
\newblock {Regular simplices and Gaussian samples.}
\newblock {\em {Discrete Comput. Geom.}}, 11(2):141--147, 1994.

\bibitem{chevet}
S.~Chevet.
\newblock Processus {G}aussiens et volumes mixtes.
\newblock {\em Z. Wahrscheinlichkeitstheorie und Verw. Gebiete}, 36(1):47--65,
  1976.

\bibitem{rE14}
R.~Eldan.
\newblock Volumetric properties of the convex hull of an {$n$}-dimensional
  {B}rownian motion.
\newblock {\em Electron. J. Probab.}, 19:no. 45, 34, 2014.

\bibitem{GHS01}
F.~Gao, D.~Hug, and R.~Schneider.
\newblock Intrinsic volumes and polar sets in spherical space.
\newblock {\em Math. Notae}, 41:159--176 (2003), 2001/02.
\newblock Homage to Luis Santal\'{o}. Vol. 1 (Spanish).

\bibitem{GV01}
F.~Gao and R.~A. Vitale.
\newblock Intrinsic volumes of the {B}rownian motion body.
\newblock {\em Discrete Comput. Geom.}, 26(1):41--50, 2001.

\bibitem{glasauer_phd}
S.~Glasauer.
\newblock {Integralgeometrie konvexer K\"orper im sph\"arischen Raum}.
\newblock PhD Thesis, University of Freiburg. Available at:
  \url{http://www.hs-augsburg.de/~glasauer/publ/diss.pdf}, 1995.

\bibitem{GNP17}
L.~Goldstein, I.~Nourdin, and G.~Peccati.
\newblock Gaussian phase transitions and conic intrinsic volumes: {S}teining
  the {S}teiner formula.
\newblock {\em Ann. Appl. Probab.}, 27(1):1--47, 2017.

\bibitem{bG68}
B.~Gr\"{u}nbaum.
\newblock Grassmann angles of convex polytopes.
\newblock {\em Acta Math.}, 121:293--302, 1968.

\bibitem{gH43}
G.~Herglotz.
\newblock \"{U}ber die {S}teinersche {F}ormel f\"{u}r {P}arallelfl\"{a}chen.
\newblock {\em Abh. Math. Sem. Hansischen Univ.}, 15:165--177, 1943.

\bibitem{convex_hull_sphere}
Z.~Kabluchko, A.~Marynych, D.~Temesvari, and C.~Th{\"a}le.
\newblock Cones generated by random points on half-spheres and convex hulls of
  {P}oisson point processes.
\newblock {\em Probab. Theory and Related Fields, to appear}, 2018.
\newblock Preprint at arXiv: 1801.08008.

\bibitem{KVZ17a}
Z.~Kabluchko, V.~Vysotsky, and D.~Zaporozhets.
\newblock Convex hulls of random walks, hyperplane arrangements, and {W}eyl
  chambers.
\newblock {\em Geom. Funct. Anal.}, 27(4):880--918, 2017.

\bibitem{KZ16}
Z.~Kabluchko and D.~Zaporozhets.
\newblock Intrinsic volumes of {S}obolev balls with applications to {B}rownian
  convex hulls.
\newblock {\em Trans. Amer. Math. Soc.}, 368(12):8873--8899, 2016.

\bibitem{KZ18}
Z.~Kabluchko and D.~Zaporozhets.
\newblock Angles of the {G}aussian simplex.
\newblock {\em Zap. Nauchn. Sem. S.-Peterburg. Otdel. Mat. Inst. Steklov.
  (POMI)}, 476(Geometriya i Topologiya. 13):79--91, 2018.

\bibitem{KZ19}
Z.~Kabluchko and D.~Zaporozhets.
\newblock Expected volumes of {G}aussian polytopes, external angles, and
  multiple order statistics.
\newblock {\em Trans. Amer. Math. Soc.}, 372(3):1709--1733, 2019.

\bibitem{MR1608265}
D.~A. Klain and G.-C. Rota.
\newblock {\em Introduction to geometric probability}.
\newblock Lezioni Lincee. [Lincei Lectures]. Cambridge University Press,
  Cambridge, 1997.

\bibitem{KS11}
C.~J. Klivans and E.~Swartz.
\newblock Projection volumes of hyperplane arrangements.
\newblock {\em Discrete Comput. Geom.}, 46(3):417--426, 2011.

\bibitem{MT14}
M.~B. McCoy and J.~A. Tropp.
\newblock From {S}teiner formulas for cones to concentration of intrinsic
  volumes.
\newblock {\em Discrete Comput. Geom.}, 51(4):926--963, 2014.

\bibitem{rR70}
R.~T. Rockafellar.
\newblock {\em Convex analysis}.
\newblock Princeton Mathematical Series, No. 28. Princeton University Press,
  Princeton, N.J., 1970.

\bibitem{lS76}
L.~A. Santal\'{o}.
\newblock {\em Integral geometry and geometric probability}.
\newblock Addison-Wesley Publishing Co., Reading, Mass.-London-Amsterdam, 1976.
\newblock With a foreword by Mark Kac, Encyclopedia of Mathematics and its
  Applications, Vol. 1.

\bibitem{SchneiderBook}
R.~Schneider.
\newblock {\em Convex bodies: the {B}runn-{M}inkowski theory}, volume~44 of
  {\em Encyclopedia of Mathematics and its Applications}.
\newblock Cambridge University Press, Cambridge, 1993.

\bibitem{SW08}
R.~Schneider and W.~Weil.
\newblock {\em Stochastic and integral geometry}.
\newblock Probability and its Applications (New York). Springer-Verlag, Berlin,
  2008.

\bibitem{vS79}
V.~N. Sudakov.
\newblock Geometric problems in the theory of infinite-dimensional probability
  distributions.
\newblock {\em Proc. Steklov Inst. Math.}, (2):i--v, 1--178, 1979.
\newblock Cover to cover translation of Trudy Mat. Inst. Steklov {{\bf{1}}41}
  (1976).

\bibitem{bT85}
B.~S. Tsirelson.
\newblock A geometric approach to maximum likelihood estimation for an
  infinite-dimensional {G}aussian location. {II}.
\newblock {\em Teor. Veroyatnost. i Primenen.}, 30(4):772--779, 1985.

\bibitem{vitale01}
R.~A. Vitale.
\newblock Intrinsic volumes and {G}aussian processes.
\newblock {\em Adv. in Appl. Probab.}, 33(2):354--364, 2001.

\bibitem{vitale08}
R.~A. Vitale.
\newblock On the {G}aussian representation of intrinsic volumes.
\newblock {\em Statist. Probab. Lett.}, 78(10):1246--1249, 2008.

\bibitem{vitale10}
R.~A. Vitale.
\newblock Convex bodies and {G}aussian processes.
\newblock {\em Image Anal. Stereol.}, 29(1):13--18, 2010.

\bibitem{jW62}
J.~G. Wendel.
\newblock A problem in geometric probability.
\newblock {\em Math. Scand.}, 11:109--111, 1962.

\end{thebibliography}

\addtocontents{toc}{\protect\setcounter{tocdepth}{-1}}

\end{document}